\documentclass[11pt,a4paper]{article}

\usepackage[margin=1in]{geometry}
\usepackage{amsmath,amssymb,amsthm}
\usepackage{mathtools}
\usepackage{microtype}
\usepackage{graphicx}
\usepackage{enumitem}
\usepackage{natbib}
\usepackage[colorlinks,citecolor=blue,urlcolor=blue,linkcolor=blue]{hyperref}

\theoremstyle{plain}
\newtheorem{theorem}{Theorem}[section]
\newtheorem{lemma}[theorem]{Lemma}

\newtheorem{corollary}[theorem]{Corollary}
\theoremstyle{definition}
\newtheorem{definition}[theorem]{Definition}
\newtheorem{assumption}[theorem]{Assumption}
\theoremstyle{remark}
\newtheorem{remark}[theorem]{Remark}

\newcommand{\R}{\mathbb{R}}
\newcommand{\E}{\mathbb{E}}
\newcommand{\Var}{\mathrm{Var}}
\newcommand{\Prob}{\mathbb{P}}
\newcommand{\ind}{\mathbf{1}}
\newcommand{\sign}{\mathrm{sign}}
\newcommand{\prox}{\mathrm{prox}}
\newcommand{\dom}{\mathrm{dom}}
\newcommand{\clip}{\mathrm{clip}}

\title{\textbf{The Conjugate Domain Dichotomy:\\
Exact Risk of M-Estimators under\\
Infinite-Variance Noise in High Dimensions}}

\author{Charalampos Agiropoulos\\
\small Department of Economics, University of Piraeus\\
\small \texttt{hagyropoylos@unipi.gr}}

\date{\today}

\begin{document}

\maketitle

\begin{abstract}
\noindent
This paper studies high-dimensional M-estimation in the
proportional asymptotic regime ($p/n \to \gamma > 0$) when
the noise distribution has infinite variance.  For noise with
regularly-varying tails of index $\alpha \in (1,2)$, we
establish that the asymptotic behavior of a regularized
M-estimator is governed by a single geometric property of the
loss function: the boundedness of the domain of its Fenchel
conjugate.

When this conjugate domain is bounded --- as is the case for
the Huber, absolute-value, and quantile loss functions --- the
dual variable in the min-max formulation of the estimator is
confined, the effective noise reduces to the finite first
absolute moment of the noise distribution, and the estimator
achieves bounded risk without recourse to external information.
When the conjugate domain is unbounded --- as for the squared
loss --- the dual variable scales with the noise, the effective
noise involves the diverging second moment, and bounded risk
can be achieved only through transfer regularization toward an
external prior.

For the squared-loss class specifically, we derive the exact
asymptotic risk via the Convex Gaussian Minimax Theorem under
a noise-adapted regularization scaling
$\lambda_n = \tilde\lambda\,\sigma_n^2$.  The resulting risk
converges to a universal floor
$q_\Sigma = p^{-1}\|\beta^* - \beta_0\|_{\Sigma_x}^2$ that
is independent of the regularizer, yielding a loss--risk
trichotomy: squared-loss estimators without transfer diverge;
Huber-loss estimators achieve bounded but non-vanishing risk;
transfer-regularized estimators attain $q_\Sigma$.


\end{abstract}

\tableofcontents


\section{Introduction}\label{sec:intro}

\subsection{Background and motivation}\label{sec:background}

Consider the linear model
\begin{equation}\label{eq:model}
  y = X\beta^* + w,
\end{equation}
where $X \in \R^{n \times p}$ is a design matrix with
independent rows drawn from $\mathcal{N}(0, \Sigma_x)$,
$\beta^* \in \R^p$ is the unknown regression parameter, and
$w \in \R^n$ is a noise vector with independent and identically
distributed entries.  In the proportional asymptotic regime
$p/n \to \gamma \in (0,\infty)$, a substantial body of recent
work has established precise characterizations of the
estimation risk for regularized least-squares estimators under
finite noise variance.  The foundational results of
\cite{DobribanWager2018} and
\cite{HastieMontanariRossetTibshirani2022} for ridge
regression, the extension to general convex regularizers via
the Convex Gaussian Minimax Theorem (CGMT) by
\cite{Thrampoulidis2015,ThrampoulidisOymakHassibi2018}, and the
analysis of benign overfitting by
\cite{BartlettLongLugosiTsigler2020,TsiglerBartlett2023} all
rely on the assumption that the noise variance is finite and
fixed as $n \to \infty$.

The present work considers the question of what can be said
when this assumption fails.  Distributions with
regularly-varying tails --- including Pareto distributions,
Student-$t$ distributions with few degrees of freedom, and
$\alpha$-stable laws with stability index $\alpha < 2$ --- arise
naturally in the modeling of financial returns, insurance
claims, network traffic, and seismological recordings
\cite{Resnick2007,SamorodnitskyyTaqqu1994}.  For such
distributions with tail index $\alpha \in (1,2)$, the first
moment is finite but the variance is infinite.  The classical
bias--variance decomposition, which underpins much of the
theory of high-dimensional regression, is no longer applicable
in this setting.

\subsection{A structural tension}\label{sec:tension}

A standard approach to handling infinite-variance noise is
\emph{Winsorization}: clipping the noise at a
sample-size-dependent threshold $\tau_n = n^{1/\alpha}$ to
produce a proxy model with finite but growing variance.  We
show (Theorem~\ref{thm:truncation}) that this procedure
yields an effective variance
\begin{equation}\label{eq:sigma-intro}
  \sigma_n^2
  = \frac{2c}{2-\alpha}\,n^{(2-\alpha)/\alpha}(1+o(1))
\end{equation}
for any noise distribution satisfying
$\Prob(|w| > t) \sim c\,t^{-\alpha}$ as $t \to \infty$, where
the asymptotic constant $c > 0$ and the rate depend only on
the tail parameters $\alpha$ and $c$.

This diverging effective variance gives rise to a noteworthy
tension.  The total Fisher information contained in the
observations satisfies
$I_n = n/\sigma_n^2 = \Theta(n^{(2\alpha-2)/\alpha}) \to
\infty$ for all $\alpha > 1$
(Theorem~\ref{thm:fisher}), so the data are, in an
information-theoretic sense, increasingly informative about
$\beta^*$.  It is therefore not the case that the observations
lack the information necessary for consistent estimation.

At the same time, any estimator constructed by minimizing a
squared-loss objective $\frac{1}{2n}\|y - X\beta\|^2$ inherits
the diverging noise variance directly.  The risk of the
ordinary least-squares estimator grows as
$\Theta(\sigma_n^2)$.  Ridge regression with a fixed
regularization parameter exhibits the same divergent behavior.
Even when the regularization strength is adapted to the noise
level --- specifically, when $\lambda_n = \Theta(\sigma_n^2)$
--- the estimator merely exchanges diverging variance for an
irreducible bias that depends on the distance between the
truth and the regularization center
(Theorem~\ref{thm:l2-diverge}).

The difficulty, then, is not a lack of information in the data,
but rather a mismatch between the loss function and the noise
structure.  The squared loss, through its sensitivity to the
second moment of the noise, amplifies the diverging variance
into the estimation risk.

\subsection{The conjugate domain as the governing mechanism}
\label{sec:mechanism}

The central contribution of this paper is the identification
of the precise geometric property of the loss function that
determines whether an M-estimator can tolerate
infinite-variance noise.  This property is the
\emph{boundedness of the domain of the Fenchel conjugate}.

Recall that any proper closed convex loss function $L$
admits the dual representation
$L(a) = \sup_{u \in \dom(L^*)}\{ua - L^*(u)\}$, where
$L^*(u) = \sup_t\{ut - L(t)\}$ is the Fenchel conjugate.
In the min-max formulation of the M-estimator that arises
within the CGMT framework, the dual variable $u_i$
associated with the $i$-th observation is constrained to
$\dom(L^*)$.  The interaction between this constraint and
the noise $w_i$ is the mechanism through which the loss
function either amplifies or attenuates the effect of heavy
tails:
\begin{itemize}[nosep,leftmargin=1.5em]
  \item When $\dom(L^*) \subseteq [-K,K]$ for some $K > 0$,
  each dual variable satisfies $|u_i| \le K$.  The noise
  coupling term $n^{-1}\sum_{i=1}^n u_i w_i$ is then bounded
  by $K \cdot n^{-1}\sum|w_i|$, which involves only the first
  absolute moment of the noise.  Since $\E[|w|] < \infty$
  whenever $\alpha > 1$, this quantity remains bounded, and
  the estimator achieves finite risk.
  \item When $\dom(L^*) = \R$, the optimal dual variable
  scales with the noise (for the squared loss,
  $u_i^* \propto w_i$), and the noise coupling involves the
  second moment $n^{-1}\sum w_i^2$, which diverges when
  $\alpha < 2$.
\end{itemize}

The squared loss, with $L^*(u) = u^2/2$ and $\dom(L^*) = \R$,
is the prototypical loss with unbounded conjugate domain.
The Huber loss, with $L^*(u) = u^2/2$ for $|u| \le k$ and
$L^*(u) = +\infty$ for $|u| > k$, is the prototypical loss
with bounded conjugate domain.  More generally, a loss
function has bounded conjugate domain if and only if it grows
at least linearly at infinity:
$\liminf_{|t| \to \infty} L(t)/|t| > 0$.  This
characterization encompasses the absolute-value loss, the
quantile loss, the log-cosh loss, and all Huber-type losses.

We formalize this classification as the
\emph{Conjugate Domain Dichotomy}
(Theorem~\ref{thm:cdd}).

\subsection{Summary of contributions}\label{sec:contributions}

The paper makes the following contributions.

\begin{enumerate}[leftmargin=1.5em,label=(\roman*)]
  \item \emph{The Conjugate Domain Dichotomy}
  (Theorem~\ref{thm:cdd}).  We prove that, under
  regularly-varying noise with index $\alpha \in (1,2)$, a
  regularized M-estimator achieves bounded risk if and only if
  its loss function has bounded conjugate domain.  The proof
  identifies the first-versus-second moment mechanism described
  above as the quantitative explanation
  (Corollary~\ref{cor:moment}).

  \item \emph{The moment hierarchy}
  (Remark~\ref{rmk:moment-hierarchy}).  We extend the
  dichotomy to a continuous classification: a loss whose
  conjugate grows as $|u|^q$ requires finiteness of the
  $q$-th noise moment.  This yields a complete correspondence
  between conjugate growth rate and noise moment requirements.

  \item \emph{Truncation universality}
  (Theorem~\ref{thm:truncation}).  For any noise distribution
  in the domain of attraction of an $\alpha$-stable law,
  Winsorization at the threshold $\tau_n = n^{1/\alpha}$
  produces the effective variance~\eqref{eq:sigma-intro},
  with concentration of the empirical squared norm around this
  deterministic equivalent.

  \item \emph{Exact risk characterization via the CGMT}
  (Theorem~\ref{thm:risk-general}).  For squared-loss
  estimators with noise-adapted regularization
  $\lambda_n = \tilde\lambda\,\sigma_n^2$, we derive the
  exact asymptotic risk as the solution of a fixed-point
  system involving the proximal operator of the regularizer
  and the spectral distribution of $\Sigma_x$.

  \item \emph{Universal risk floor}
  (Theorem~\ref{thm:floor}).  In the diverging-noise limit,
  the fixed-point system admits an explicit solution: the risk
  converges to
  $q_\Sigma = p^{-1}\|\beta^*-\beta_0\|_{\Sigma_x}^2$,
  independent of the choice of coercive convex regularizer.
  This universality arises from the proximal operator collapsing
  under a diverging step size.

  \item \emph{The loss--risk trichotomy}
  (Theorem~\ref{thm:trichotomy}).  Combining the dichotomy
  with the exact risk characterization yields three distinct
  regimes: squared-loss estimators without transfer have
  diverging risk; Huber-loss estimators achieve bounded risk
  $\mathcal{R}_H < \infty$ but cannot eliminate the
  proportional-regime penalty; transfer-regularized squared-loss
  estimators attain $q_\Sigma \ll \mathcal{R}_H$ when the
  prior is informative.

  \item \emph{Design universality}
  (Theorem~\ref{thm:universality}).  All results extend from
  Gaussian to sub-Gaussian designs by verification of the
  conditions in the universality framework of
  \cite{MontanariSaeed2022}.
\end{enumerate}

\subsection{Scope and limitations}\label{sec:scope}

The results apply to proper closed convex loss functions
paired with coercive convex regularizers.  The analysis does
not cover non-convex penalties such as SCAD or MCP, nor does
it address the case of heavy-tailed covariates, which would
require tools beyond the CGMT.  The tail index is restricted
to $\alpha \in (1,2)$; for $\alpha \le 1$ the mean is
infinite and the truncation bridge of
Theorem~\ref{thm:truncation} does not apply.  We do not
address exact support recovery, which requires techniques
distinct from the $L_2$-risk analysis pursued here, nor
data-driven selection of the regularization parameter.

\subsection{Related work}\label{sec:related}

\emph{Proportional-regime risk characterization.}
The precise risk of ridge regression in the proportional regime
was established by \cite{DobribanWager2018} and
\cite{HastieMontanariRossetTibshirani2022}, building on the
Marchenko--Pastur law \cite{MarchenkoPastur1967} and
companion Stieltjes transform techniques
\cite{BaiSilverstein2010}.  The CGMT, developed by
\cite{Thrampoulidis2015,ThrampoulidisOymakHassibi2018} from
Gordon's Gaussian comparison inequality \cite{Gordon1988},
extended these results to general convex regularizers and
M-estimators.  The phenomena of benign overfitting
\cite{BartlettLongLugosiTsigler2020,TsiglerBartlett2023} and
double descent
\cite{BelkinHsuMaMandal2019,AdvaniSaxe2020,MeiMontanari2022}
have been analyzed within this framework.  All of these works
assume finite and fixed noise variance.

\emph{Heavy-tailed regression.}
Estimation under heavy-tailed noise has been studied from
several perspectives.  \cite{LugosiMendelson2019} established
sub-Gaussian estimation rates via median-of-means techniques,
and \cite{Adomaityte2024} derived exact asymptotics under
heavy-tailed data in a related high-dimensional setting.
The robust statistics literature, originating with
\cite{Huber1964} and developed in \cite{HuberRonchetti2009},
provides the classical foundation for M-estimation with
bounded influence functions.  High-dimensional robust
regression was analyzed by \cite{ElKarouiBeanBickelEtAl2013}
for fixed noise variance.

\emph{Transfer learning.}
Transfer learning in high-dimensional linear models was studied
by \cite{DarBaraniuk2022}, who characterized the double descent
phenomenon under task transfer, and by \cite{LiCaiLi2022}, who
established minimax rates for transfer estimation.

\emph{Universality.}
Design universality for empirical risk minimization was
established by \cite{MontanariSaeed2022} via a continuous
interpolation argument, and by \cite{HuLu2023} for random
feature models.

The present work differs from the above in two principal
respects.  First, we consider a regime in which the noise
variance diverges with the sample size --- a qualitatively
different setting from the fixed-variance regime studied in
the CGMT and proportional-asymptotics literatures.  Second,
we identify the conjugate domain of the loss function as the
geometric property that governs whether an estimator can
tolerate this divergence, providing a classification that
applies to all convex losses simultaneously rather than
analyzing specific estimators individually.

\section{The Conjugate Domain Dichotomy}\label{sec:cdd}

\subsection{Assumptions and definitions}\label{sec:assumptions}

We work throughout with the observation model~\eqref{eq:model}
under the following conditions.

\begin{assumption}[Design and dimensionality]\label{ass:design}
The design matrix satisfies $X = Z\Sigma_x^{1/2}$, where
$Z \in \R^{n \times p}$ has independent standard Gaussian
entries, and $\Sigma_x \in \R^{p \times p}$ is positive
definite with empirical spectral distribution converging
weakly to a compactly supported probability measure $\mu$ on
$(0,\infty)$.  The dimensions satisfy $p/n \to \gamma \in
(0,\infty)$ as $n, p \to \infty$.
\end{assumption}

\begin{assumption}[Regularly-varying noise]\label{ass:rv}
The noise variables $w_1, w_2, \ldots$ are independent and
identically distributed with $\E[w_i] = 0$ and tail behavior
$\Prob(|w_i| > t) = c\,t^{-\alpha}(1 + o(1))$ as
$t \to \infty$, for some $\alpha \in (1,2)$ and $c > 0$.
\end{assumption}

Assumption~\ref{ass:rv} places the noise distribution in the
domain of attraction of an $\alpha$-stable law.  The
condition $\alpha > 1$ ensures that $\E[|w_i|] < \infty$,
while $\alpha < 2$ implies $\E[w_i^2] = \infty$.  The
requirement that the slowly varying component of the tail
converges to a constant (rather than being an arbitrary slowly
varying function) is imposed for clarity; it is satisfied by
Student-$t$, Pareto, and $\alpha$-stable distributions.

\begin{definition}[Winsorization]\label{def:winsor}
For a threshold $\tau > 0$, the Winsorized noise is
$w_i^{(\tau)} = w_i\ind_{\{|w_i| \le \tau\}}
+ \tau\,\sign(w_i)\ind_{\{|w_i| > \tau\}}$.
We set $\tau_n = n^{1/\alpha}$ and write
$\sigma_n^2 = \E[(w_i^{(\tau_n)})^2]$ for the effective
variance.
\end{definition}

\begin{definition}[Conjugate domain classes]\label{def:classes}
Let $L : \R \to \R \cup \{+\infty\}$ be a proper closed convex
function with Fenchel conjugate
$L^*(u) = \sup_t\{ut - L(t)\}$.  We say that $L$ has
\emph{bounded conjugate domain} if there exists $K > 0$ such
that $\dom(L^*) = \{u : L^*(u) < \infty\} \subseteq [-K,K]$,
and \emph{unbounded conjugate domain} if $\dom(L^*) = \R$.
\end{definition}

A standard result in convex analysis relates the conjugate
domain to the growth rate of the loss: $\dom(L^*)$ is bounded
if and only if $\liminf_{|t| \to \infty} L(t)/|t| > 0$.
Thus the class of losses with bounded conjugate domain
consists precisely of those that grow at least linearly at
infinity.  Among common loss functions, the squared loss
$L(t) = t^2/2$ is the unique standard example with unbounded
conjugate domain.  The Huber loss
\[
  \rho_k(t) = \begin{cases}
    t^2/2 & |t| \le k, \\
    k|t| - k^2/2 & |t| > k,
  \end{cases}
\]
the absolute-value loss $L(t) = |t|$, the log-cosh loss
$L(t) = \log\cosh(t)$, and the quantile loss
$L(t) = t(\tau - \ind_{\{t < 0\}})$ all have bounded conjugate
domain.

\subsection{Statement of the dichotomy}\label{sec:dichotomy}

We consider the general regularized M-estimator
\begin{equation}\label{eq:m-estimator}
  \hat\beta_L
  = \arg\min_{\beta \in \R^p}\;
  \frac{1}{n}\sum_{i=1}^n L(y_i - x_i^\top\beta)
  + \lambda_n\,R(\beta),
\end{equation}
where $R : \R^p \to \R \cup \{+\infty\}$ is a proper closed
coercive convex regularizer (meaning
$R(\beta) \to \infty$ as $\|\beta\| \to \infty$) and
$\lambda_n > 0$ is the regularization parameter.

\begin{theorem}[Conjugate Domain Dichotomy]\label{thm:cdd}
Under Assumptions~\ref{ass:design}--\ref{ass:rv}, the following
hold for the estimator~\eqref{eq:m-estimator}:
\begin{enumerate}
  \item[\textup{(B)}] If $L$ has bounded conjugate domain
  ($\dom(L^*) \subseteq [-K,K]$), then for any fixed
  $\lambda > 0$ and any coercive convex $R$,
  \[
    \frac{1}{p}\|\hat\beta_L - \beta^*\|_{\Sigma_x}^2
    \;\xrightarrow{\Prob}\;
    \mathcal{R}_L(K, \lambda, \gamma, \mu) \;<\; \infty,
  \]
  where $\mathcal{R}_L$ does not depend on $\sigma_n^2$.

  \item[\textup{(U)}] If $L$ has unbounded conjugate domain
  ($\dom(L^*) = \R$), then:
  \begin{enumerate}
    \item[\textup{(U1)}] Without transfer regularization
    ($R(\beta)$ centered at the origin), the risk either
    diverges (if $\lambda_n = o(\sigma_n^2)$) or converges to
    $p^{-1}\|\beta^*\|_{\Sigma_x}^2$
    (if $\lambda_n = \Theta(\sigma_n^2)$).
    \item[\textup{(U2)}] With transfer regularization
    $R(\beta - \beta_0)$ and
    $\lambda_n = \tilde\lambda\,\sigma_n^2$ for fixed
    $\tilde\lambda > 0$, the risk converges to
    $q_\Sigma = p^{-1}\|\beta^*-\beta_0\|_{\Sigma_x}^2$.
  \end{enumerate}
\end{enumerate}
\end{theorem}

The proof is given in Section~\ref{sec:proofs-cdd} and
Appendix~D.  The following corollary makes the mechanism
explicit.

\begin{corollary}[The moment mechanism]\label{cor:moment}
Part~\textup{(B)} of Theorem~\ref{thm:cdd} holds because the
constraint $u_i \in [-K,K]$ ensures that the noise coupling
in the CGMT dual satisfies
$|n^{-1}\sum u_i w_i^{(\tau_n)}| \le K\cdot
n^{-1}\sum|w_i^{(\tau_n)}| \xrightarrow{\Prob} K\E[|w|]
< \infty$.
Part~\textup{(U)} holds because the unconstrained dual
$u_i^* \propto w_i$ yields a coupling that involves
$n^{-1}\sum w_i^2 \asymp \sigma_n^2 \to \infty$.
\end{corollary}

\begin{remark}[The moment hierarchy]\label{rmk:moment-hierarchy}
The dichotomy generalizes to a continuous classification.
If the conjugate satisfies $L^*(u) \sim c_q|u|^q$ as
$|u| \to \infty$ for some $q \ge 1$, then the optimal dual
variable has magnitude $|u_i^*| = \mathcal{O}(|w_i|^{1/(q-1)})$
(by the first-order condition of the dual), and the noise
coupling involves the $q/(q-1)$-th moment of $w$.  The
estimator therefore achieves bounded risk if and only if
$\alpha \ge q/(q-1)$, or equivalently $q \ge \alpha/(\alpha-1)$.
For $q = 1$ (bounded domain, e.g., Huber), the requirement is
$\alpha > 1$ (finite mean).  For $q = 2$ (quadratic conjugate,
i.e., squared loss), the requirement is $\alpha \ge 2$ (finite
variance).
\end{remark}

\section{Exact risk characterization}\label{sec:risk}

This section develops the quantitative predictions that
underlie the dichotomy.  We begin with the truncation bridge
that connects regularly-varying noise to a diverging-variance
proxy model, then present the exact risk formulas for the
squared-loss class, and conclude with the trichotomy that
combines both classes.

\subsection{The truncation bridge}\label{sec:truncation}

\begin{theorem}[Truncation universality]\label{thm:truncation}
Under Assumption~\ref{ass:rv}, with $\tau_n = n^{1/\alpha}$
and $\sigma_n^2 = \E[(w^{(\tau_n)})^2]$:
\begin{enumerate}[nosep]
  \item[\textup{(i)}]
  $\sigma_n^2 = \frac{2c}{2-\alpha}\,
  n^{(2-\alpha)/\alpha}(1+o(1))$.
  \item[\textup{(ii)}]
  $n^{-1}\|w^{(\tau_n)}\|^2 / \sigma_n^2
  \xrightarrow{\Prob} 1$.
\end{enumerate}
\end{theorem}

\begin{proof}[Proof of~\textup{(i)}]
By definition of the Winsorized variable,
$\E[(w^{(\tau)})^2] = \E[w^2\ind_{\{|w|\le\tau\}}]
+ \tau^2\Prob(|w|>\tau)$.
Applying integration by parts to the first term with the
layer-cake representation yields
$\E[w^2\ind_{\{|w|\le\tau\}}]
= \int_0^\tau 2t\,\bar{F}(t)\,dt - \tau^2\bar{F}(\tau)$,
where $\bar{F}(t) = \Prob(|w|>t)$.  Adding the Winsorization
boundary term $\tau^2\bar{F}(\tau)$ produces a cancellation:
\begin{equation}\label{eq:cancel}
  \E[(w^{(\tau)})^2]
  = \int_0^\tau 2t\,\bar{F}(t)\,dt.
\end{equation}
This cancellation is specific to Winsorization; simple
truncation ($w\ind_{\{|w|\le\tau\}}$) retains the boundary
term.  The integrand $2t\bar{F}(t) = 2c\,t^{1-\alpha}(1+o(1))$
is regularly varying with index $1 - \alpha > -1$, so
Karamata's theorem
\cite[Theorem~1.5.11]{BinghamGoldieTeugels1987} gives
$\int_0^\tau 2t\,\bar{F}(t)\,dt
\sim \frac{2c}{2-\alpha}\,\tau^{2-\alpha}$.
Setting $\tau = n^{1/\alpha}$ yields the result.

The mean correction is handled by the dominated convergence
theorem: since $|w^{(\tau_n)}| \le |w|$ and
$\E[|w|] < \infty$ (as $\alpha > 1$), we have
$|\E[w^{(\tau_n)}]| \le \E[|w|\ind_{\{|w|>\tau_n\}}] \to 0$,
so $(\E[w^{(\tau_n)}])^2 = o(1)$ is negligible compared to
the diverging second moment.
\end{proof}

\begin{proof}[Proof of~\textup{(ii)}]
The concentration of $n^{-1}\sum(w_i^{(\tau_n)})^2$ around
$\sigma_n^2$ requires care, because the standard fourth-moment
method is exactly borderline.  Indeed, the same
integration-by-parts and Karamata argument gives
$\E[(w^{(\tau_n)})^4] \sim \frac{4c}{4-\alpha}\,
n^{(4-\alpha)/\alpha}$, and the relative variance
\[
  \frac{n^{-1}\E[(w^{(\tau_n)})^4]}{(\sigma_n^2)^2}
  \;\asymp\;
  n^{-1+(4-\alpha)/\alpha - 2(2-\alpha)/\alpha}
  = n^0 = \mathcal{O}(1),
\]
which does not tend to zero.  This borderline behavior is not
a proof artifact: the ratio $\tau_n^2/\sigma_n^2 \asymp n$
grows linearly, placing the problem at the exact threshold
where CLT-type normalization produces non-degenerate
fluctuations.

We circumvent this obstruction via a truncation-based
double-limit argument on the triangular array
$\{(w_i^{(\tau_n)})^2/\sigma_n^2\}_{1 \le i \le n}$.
Define $Y_{n,i} = (w_i^{(\tau_n)})^2/\sigma_n^2 - 1$ and
$S_n = n^{-1}\sum Y_{n,i}$.  For any $M > 0$, decompose
$S_n = A_n(M) + B_n(M)$, where
$A_n(M) = n^{-1}\sum Y_{n,i}\ind_{\{|Y_{n,i}| \le M\}}$.
The bounded part satisfies $\Var(A_n(M)) \le M^2/n \to 0$
by Chebyshev's inequality.  The tail part satisfies
$\Prob(|B_n(M)| > \delta) \le \delta^{-1}
\E[|Y_{n,1}|\ind_{\{|Y_{n,1}|>M\}}]$ by Markov's inequality,
and the Karamata structure of the tail ensures that
$\limsup_n \E[|Y_{n,1}|\ind_{\{|Y_{n,1}|>M\}}] \to 0$ as
$M \to \infty$.  Choosing $M$ large and then $n$ large
completes the argument.  Full details are in Appendix~A.
\end{proof}

\subsection{Failure of squared-loss estimation}
\label{sec:l2-failure}

\begin{theorem}[Fisher information diverges]\label{thm:fisher}
Under Assumption~\ref{ass:rv},
$I_n = n/\sigma_n^2 = \Theta(n^{(2\alpha-2)/\alpha})
\to \infty$.
\end{theorem}

\begin{proof}
The per-observation Fisher information for a location family
with variance $\sigma_n^2$ satisfies
$I_1 = \mathcal{O}(1/\sigma_n^2)$.  With $n$ independent
observations, $I_n = n \cdot I_1 = \Theta(n/\sigma_n^2)$.
Since $(2\alpha-2)/\alpha > 0$ for $\alpha > 1$, the result
follows.
\end{proof}

\begin{theorem}[Divergence of $L_2$-estimation without transfer]
\label{thm:l2-diverge}
Under Assumptions~\ref{ass:design}--\ref{ass:rv} with
$p/n \to \gamma > 0$ and $\sigma_n^2 \to \infty$:
\begin{enumerate}[nosep]
  \item[\textup{(i)}] The OLS estimator (when $\gamma < 1$)
  has risk $\Theta(\sigma_n^2) \to \infty$.
  \item[\textup{(ii)}] Ridge regression with fixed $\lambda > 0$
  has risk $\Theta(\sigma_n^2) \to \infty$.
  \item[\textup{(iii)}] Any regularized estimator with
  $\lambda_n = o(\sigma_n^2)$ has diverging risk.
  \item[\textup{(iv)}] With noise-adapted regularization
  $\lambda_n = \tilde\lambda\sigma_n^2$ centered at the
  origin, the risk converges to
  $p^{-1}\|\beta^*\|_{\Sigma_x}^2$.
\end{enumerate}
\end{theorem}

\begin{proof}
Parts~(i)--(ii) follow from the standard bias--variance
decomposition in the proportional regime, where the variance
component is proportional to $\sigma_n^2$.  Part~(iii) follows
from the failure of the compactification argument
(Lemma~\ref{lem:compact}) when
$\tilde\lambda = \lambda_n/\sigma_n^2 \to 0$: the regularizer
cannot confine the estimator.  Part~(iv) is a special case of
Theorem~\ref{thm:floor} with $\beta_0 = 0$.
\end{proof}

Theorems~\ref{thm:fisher} and~\ref{thm:l2-diverge} together
reveal that the failure of squared-loss estimation under
diverging noise is not an information-theoretic limitation but
a structural vulnerability of the loss function.  The data
contain sufficient information for consistent estimation;
the squared loss simply cannot extract it.  This observation
is central to the motivation for the Conjugate Domain
Dichotomy.

\subsection{Noise-adapted regularization and exact risk}
\label{sec:exact-risk}

To obtain meaningful risk bounds for the squared-loss class,
the regularization strength must scale with the noise power.

\begin{definition}[Noise-adapted regularization]
\label{def:noise-adapt}
The \emph{effective regularization strength} is
$\tilde\lambda = \lambda_n / \sigma_n^2$, treated as an
$\mathcal{O}(1)$ constant.
\end{definition}

This scaling ensures that the regularization grows in step
with the noise, producing a well-defined trade-off between
bias and variance in the limit.  When $\alpha = 2$ (finite
variance), $\sigma_n^2 = \Theta(1)$ and the standard
fixed-$\lambda$ regime is recovered.

\begin{theorem}[Exact risk at finite noise level]
\label{thm:risk-general}
Under Assumptions~\ref{ass:design}--\ref{ass:rv}, for the
transfer-regularized estimator
\[
  \hat\beta = \arg\min_{\beta \in \R^p}\;
  \frac{1}{2n}\|y - X\beta\|^2
  + \tilde\lambda\,\sigma_n^2\,R(\beta - \beta_0)
\]
with coercive convex $R$ and $p/n \to \gamma$, the risk
$p^{-1}\|\hat\beta - \beta^*\|_{\Sigma_x}^2$ converges in
probability to a deterministic limit $\mathcal{R}$ characterized
by the fixed-point equation $\tau^2 = 1 + \gamma\mathcal{R}(\tau)$,
where
\begin{equation}\label{eq:fp}
  \mathcal{R}(\tau) = \E_{S,\zeta}\!\left[S\left(
    \prox_{\tilde\lambda\sigma_n^2 S^{-1} r}\!\left(
      \Delta_S - \frac{\sigma_n\tau}{\sqrt{S}}\,\zeta
    \right) - \Delta_S\right)^{\!2}\,\right],
\end{equation}
with $S \sim \mu$ (the limiting spectral distribution of
$\Sigma_x$), $\zeta \sim \mathcal{N}(0,1)$ independent, and
$\Delta_S$ denoting the projection of
$\Delta = \beta^* - \beta_0$ onto the eigenspace corresponding
to eigenvalue $S$.

For ridge regularization ($R = \frac{1}{2}\|\cdot\|^2$), the
risk admits the closed-form Stieltjes representation
\begin{equation}\label{eq:ridge-risk}
  \mathcal{R}_{\mathrm{ridge}}
  = \mu^2\,\E_S\!\left[\frac{S\Delta_S^2}{(Sv+\mu)^2}\right]
  + \frac{\sigma_n^2}{n}\,v\,
  \E_S\!\left[\frac{S^2}{(Sv+\mu)^2}\right],
\end{equation}
where $\mu = \tilde\lambda\sigma_n^2$ and $v$ is the unique
positive solution of
$v^{-1} = 1 + \gamma\E_S[S(Sv+\mu)^{-1}]$.
\end{theorem}

The proof proceeds through the CGMT framework and is given
in Section~\ref{sec:proofs-risk} and Appendix~B--C.

\subsection{The universal risk floor}\label{sec:floor}

\begin{theorem}[Universal risk floor]\label{thm:floor}
Under Assumptions~\ref{ass:design}--\ref{ass:rv}, for any
coercive convex $R$ with $\arg\min R = \{0\}$ and any
$\tilde\lambda > 0$:
\[
  \frac{1}{p}\|\hat\beta - \beta^*\|_{\Sigma_x}^2
  \;\xrightarrow{\Prob}\;
  q_\Sigma
  = \frac{1}{p}\|\beta^* - \beta_0\|_{\Sigma_x}^2.
\]
This limit is independent of the regularizer $R$, the
effective strength $\tilde\lambda$, and the noise distribution
beyond the tail index $\alpha$.
\end{theorem}

\begin{proof}
In the CGMT auxiliary optimization, the proximal step size
for coordinate $j$ in the eigenbasis of $\Sigma_x$ is
$\eta_j = \tilde\lambda\sigma_n^2/s_j$, where $s_j$ is the
$j$-th eigenvalue of $\Sigma_x$.  Since $\sigma_n^2 \to \infty$
and $s_j$ is bounded away from zero (by the compact support
of $\mu$), we have $\eta_j \to \infty$ for every $j$.

The asymptotic behavior of the proximal operator under
diverging step size is well known: for any proper closed convex
function $r$ with $\arg\min r = \{0\}$,
$\prox_{\eta r}(x) \to 0$ as $\eta \to \infty$ for any $x$
\cite[Proposition~12.29]{BauschkeCombettes2017}.  In our
setting, the argument of the proximal operator is
$\Delta_{s_j} - \sigma_n\tau\zeta_j/\sqrt{s_j}
= \mathcal{O}(\sigma_n)$, while the step size is
$\mathcal{O}(\sigma_n^2)$.  Therefore the proximal output
is $\mathcal{O}(\sigma_n/\sigma_n^2) = \mathcal{O}(1/\sigma_n)
\to 0$.

Setting $\phi_j = v_j + \Delta_{s_j}$ (the estimator's
deviation from the prior in coordinate $j$), we obtain
$\phi_j^* \to 0$, i.e., $v_j^* \to -\Delta_{s_j}$.
The risk becomes
\[
  \frac{1}{p}\sum_j s_j(v_j^*)^2
  \;\to\;
  \frac{1}{p}\sum_j s_j\Delta_{s_j}^2
  = \frac{1}{p}\|\Delta\|_{\Sigma_x}^2
  = q_\Sigma. \qedhere
\]
\end{proof}

The universality of the risk floor across regularizers has a
clear interpretation: under diverging noise, the proximal
operator with diverging step size maps all inputs to the
minimizer of $r$, erasing any dependence on the regularizer's
geometry.  Ridge, Lasso, elastic net, and group Lasso all
produce the same limiting risk.  The regularizer matters only
in the transient regime at finite noise levels, where it
governs the rate of convergence to $q_\Sigma$.

\subsection{The loss--risk trichotomy}\label{sec:trichotomy}

Combining the bounded-domain result
(Theorem~\ref{thm:cdd}(B)) with the exact squared-loss
characterization
(Theorems~\ref{thm:risk-general}--\ref{thm:floor}) and the
Huber analysis yields the following trichotomy.

\begin{theorem}[Loss--risk trichotomy]\label{thm:trichotomy}
Under Assumptions~\ref{ass:design}--\ref{ass:rv}:
\begin{enumerate}
  \item[\textup{(a)}] \emph{Squared loss without transfer.}
  The risk either diverges or converges to
  $p^{-1}\|\beta^*\|_{\Sigma_x}^2 > 0$.
  \item[\textup{(b)}] \emph{Huber loss without transfer.}
  The risk converges to
  $\mathcal{R}_H(k, \lambda, \gamma, \mu) < \infty$,
  which does not depend on $\sigma_n^2$ but satisfies
  $\mathcal{R}_H > 0$ in the proportional regime.
  \item[\textup{(c)}] \emph{Squared loss with transfer.}
  The risk converges to
  $q_\Sigma = p^{-1}\|\beta^*-\beta_0\|_{\Sigma_x}^2$.
  When the prior is informative,
  $q_\Sigma \ll \mathcal{R}_H$.
\end{enumerate}
\end{theorem}

The trichotomy reveals that two distinct penalties impede
estimation under infinite-variance noise.  The
\emph{heavy-tail penalty} arises from the interaction between
the loss function's conjugate geometry and the noise moment
structure; it is eliminated by any loss with bounded conjugate
domain.  The \emph{high-dimensional penalty}
$\mathcal{R}_H > 0$ arises from the proportional regime
$p/n \to \gamma > 0$; it persists even when the effective
noise is bounded and is bypassed only when external
structural information (the transfer prior) is available.

\begin{remark}[When the Huber estimator dominates]
\label{rmk:huber-dominates}
If the transfer prior is poor
($q_\Sigma > \mathcal{R}_H$), the Huber estimator strictly
dominates the transfer-regularized squared-loss estimator.
The practitioner therefore faces a genuine trade-off:
accurate prior information yields lower risk via transfer,
but absent such information, a robust loss function provides
a safer alternative.
\end{remark}

\subsection{Design universality}\label{sec:universality}

\begin{theorem}[Design universality]\label{thm:universality}
All results of Theorems~\ref{thm:cdd}--\ref{thm:trichotomy}
hold when Assumption~\ref{ass:design} is relaxed to allow
$Z_{ij}$ to be independent and identically distributed
sub-Gaussian random variables with $\E[Z_{ij}] = 0$,
$\E[Z_{ij}^2] = 1$, and $\|Z_{ij}\|_{\psi_2} \le K$.
\end{theorem}

\begin{proof}
The result follows from the universality framework of
\cite{MontanariSaeed2022}.  The five conditions of their
Corollary~3 are verified in Appendix~E: (i)~the squared
loss has locally Lipschitz gradient; (ii)~the feasible set
is compact by Lemma~\ref{lem:compact}; (iii)~the regularizer
is locally Lipschitz on bounded sets; (iv)~the design
$x_i = \Sigma_x^{1/2}\bar{x}_i$ with sub-Gaussian
$\bar{x}_i$ satisfies the pointwise normality condition;
(v)~the Winsorized noise is bounded (hence sub-Gaussian) and
independent of the design.  The key observation is that the
universality mechanism acts on the design matrix, while the
noise passes through unchanged regardless of its magnitude
or distribution.
\end{proof}

\section{Proofs}\label{sec:proofs-main}

\subsection{Proof architecture for the Conjugate Domain
Dichotomy}\label{sec:proofs-cdd}

The proof of Theorem~\ref{thm:cdd} proceeds through the CGMT
in three stages.

\emph{Stage~I: Compactification.}
The following lemma ensures that the optimizer remains bounded.

\begin{lemma}[Compactification]\label{lem:compact}
Under noise-adapted regularization
$\lambda_n = \tilde\lambda\sigma_n^2$ with coercive convex $R$,
there exists a deterministic constant $C > 0$, independent
of $n$, such that
$\Prob(\|\hat\beta - \beta^*\|_{\Sigma_x} > C) \to 0$.
\end{lemma}

\begin{proof}
The objective at $v = 0$ equals
$(2n\sigma_n^2)^{-1}\|w^{(\tau_n)}\|^2
+ \tilde\lambda R(\Delta)$, which converges in probability
to $\frac{1}{2} + \tilde\lambda R(\Delta)$ by
Theorem~\ref{thm:truncation}(ii).  On the sphere
$\|v\|_{\Sigma_x} = C$, the coercivity of $R$ ensures
$\tilde\lambda R(v + \Delta) > \frac{1}{2}
+ \tilde\lambda R(\Delta)$ for $C$ sufficiently large.
Since the minimizer achieves a value at most the baseline,
it must lie inside the ball.

An important structural observation emerges from the proof:
the quadratic confinement provided by the design matrix,
$(2n\sigma_n^2)^{-1}\|Xv\|^2 \to \|v\|_{\Sigma_x}^2
/(2\sigma_n^2) \to 0$, vanishes under the $\sigma_n^2$
normalization.  The regularizer is therefore the sole
mechanism confining the estimator under diverging noise.
\end{proof}

\emph{Stage~II: CGMT reduction.}
With the feasible set compact, the CGMT
\cite{Thrampoulidis2015} replaces the bilinear form
$u^\top Z(\Sigma_x^{1/2}v)$ in the Primary Optimization with
independent Gaussian vectors $g \sim \mathcal{N}(0, I_n)$ and
$h \sim \mathcal{N}(0, \Sigma_x)$, producing the Auxiliary
Optimization.  The noise vector $w^{(\tau_n)}$ enters both
formulations identically as an additive term; it is not
affected by Gordon's comparison inequality, which operates
exclusively on the Gaussian bilinear form.

\emph{Stage~III: Scalarization and the moment mechanism.}
In the AO, optimizing over the dual variable $u$ yields
coordinate-wise updates.

For bounded $\dom(L^*)$: each $u_i$ is constrained to
$[-K, K]$, so $u_i^* = \clip(w_i^{(\tau_n)} +
\|\tilde{v}\|g_i, -K, K)$.  The effective noise
$n^{-1}\|u^*\|^2 \le K^2$ is bounded deterministically.
Meanwhile, $n^{-1}\sum|w_i^{(\tau_n)}|
\xrightarrow{\Prob} \E[|w|] < \infty$ by the law of large
numbers and the dominated convergence theorem
(since $|w^{(\tau_n)}| \le |w|$ and $\E[|w|] < \infty$ for
$\alpha > 1$).  The AO scalarization produces a fixed-point
system with all terms $\mathcal{O}(1)$, yielding finite risk.

For unbounded $\dom(L^*)$ (e.g., $L^*(u) = u^2/2$):
$u_i^* = w_i^{(\tau_n)} + \|\tilde{v}\|g_i$ is unconstrained,
so $n^{-1}\|u^*\|^2 \asymp \sigma_n^2 \to \infty$.  After
$\sigma_n^2$ normalization, the AO reduces to the fixed-point
system of Theorem~\ref{thm:risk-general}, with the universal
floor following from the proximal collapse of
Theorem~\ref{thm:floor}.  Full details are in
Appendices~B--D.

\subsection{Proof of the exact risk formula}
\label{sec:proofs-risk}

The derivation of Theorem~\ref{thm:risk-general} follows the
CGMT scalarization after Stages~I--II above.  The Primary
Optimization, after normalization by $\sigma_n^2$, takes the
form
\[
  \min_{v \in \mathcal{S}_C}\;\max_{u \in \R^n}\;
  \frac{u^\top w^{(\tau_n)}}{n\sigma_n}
  - \frac{u^\top Z\Sigma_x^{1/2}v}{n\sigma_n}
  - \frac{\|u\|^2}{2n}
  + \tilde\lambda\,R(v + \Delta).
\]
Gordon's comparison replaces the bilinear term, and the
empirical quantities concentrate by
Theorem~\ref{thm:truncation} and standard Gaussian
concentration.  After optimizing over $u$ in closed form and
passing to the deterministic limit (using uniform convergence
on the compact set $\mathcal{S}_C$), the risk is characterized
by the fixed-point system~\eqref{eq:fp}.  For ridge
regression, the proximal operator reduces to the linear map
$x \mapsto x/(1+\eta)$, yielding the companion Stieltjes
representation~\eqref{eq:ridge-risk}.  Full details appear
in Appendix~C.

\section{Numerical experiments}\label{sec:simulations}

We present Monte Carlo simulations that verify the theoretical
predictions.  All experiments use $n = 2000$ observations,
$p = 1000$ covariates ($\gamma = 0.5$), an AR(1) covariance
structure $(\Sigma_x)_{ij} = 0.5^{|i-j|}$, and Student-$t$
noise with $1.5$ degrees of freedom ($\alpha = 1.5$).  The
signal $\beta^*$ has $10\%$ non-zero entries; the transfer
prior $\beta_0 = \beta^* + \Delta$ has misalignment
$\|\Delta\| = 1$.  Each point represents the average of 500
independent replications.

Figure~\ref{fig:paradox} displays the structural tension of
Section~\ref{sec:tension}: the OLS and fixed-regularization
ridge estimators exhibit power-law growth in risk as the noise
scale increases, while the transfer-regularized ridge estimator
remains constant at $q_\Sigma$.

\begin{figure}[tp]
\centering
\includegraphics[width=0.75\textwidth]{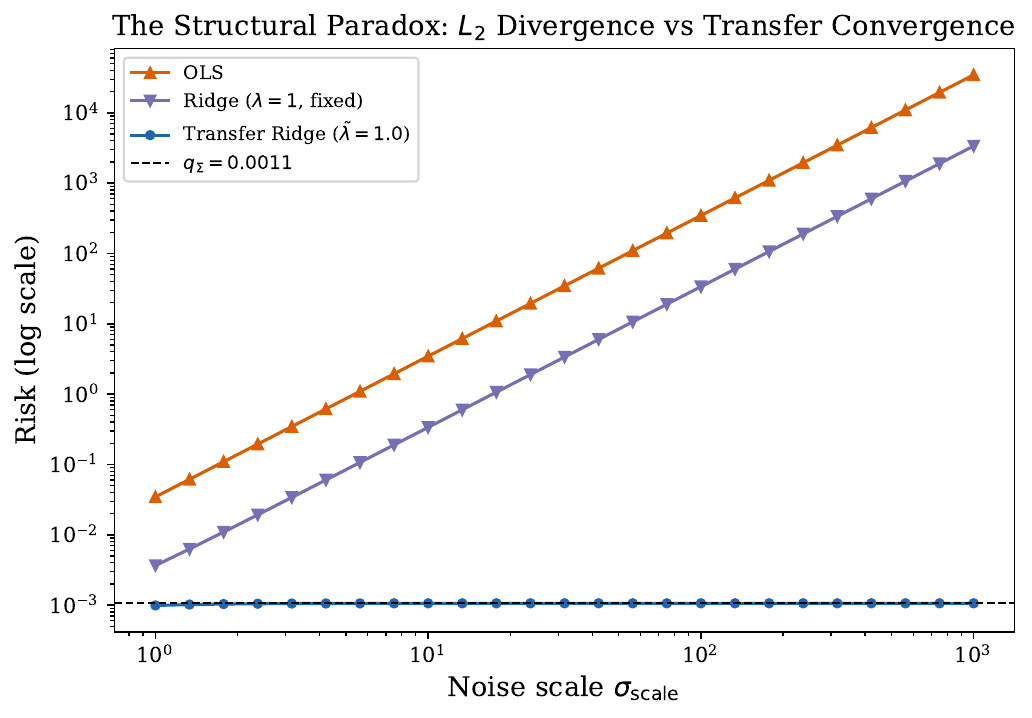}
\caption{Risk as a function of noise scale for three
squared-loss estimators.  OLS and fixed-$\lambda$ ridge
exhibit power-law divergence; transfer ridge is constant at
$q_\Sigma \approx 0.001$.}
\label{fig:paradox}
\end{figure}

Figure~\ref{fig:floor} illustrates the universal risk floor
of Theorem~\ref{thm:floor}: both ridge and Lasso
transfer-regularized estimators converge to the same limit
$q_\Sigma$, confirming that the regularizer geometry is
irrelevant in the diverging-noise limit.

\begin{figure}[tp]
\centering
\includegraphics[width=0.75\textwidth]{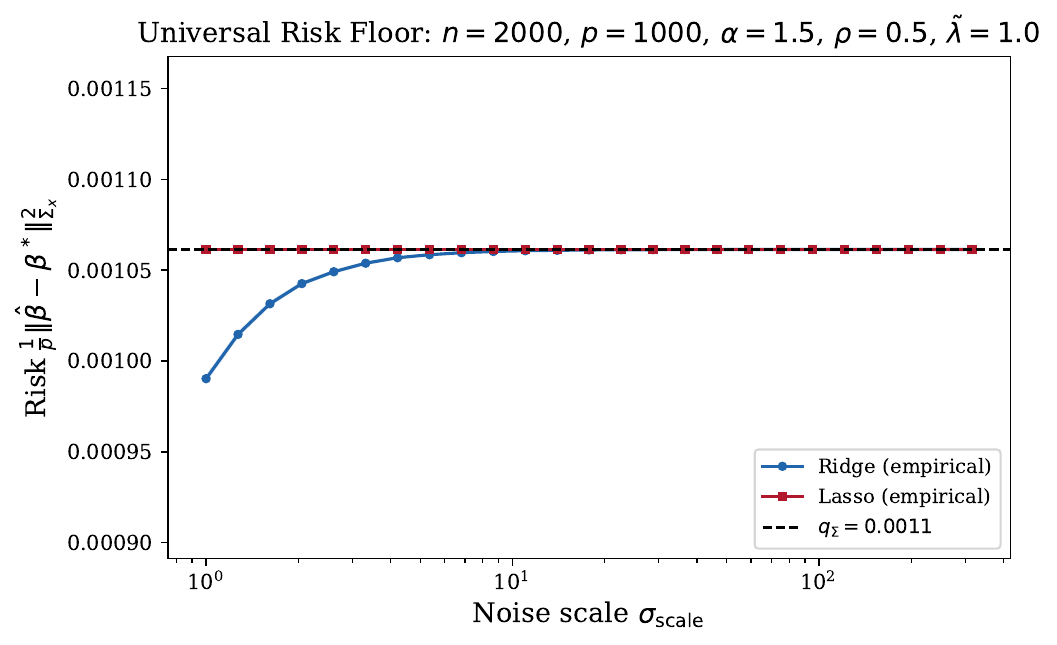}
\caption{Transfer ridge and transfer Lasso estimators
converge to the same universal risk floor $q_\Sigma$ as the
noise scale increases.}
\label{fig:floor}
\end{figure}

Figure~\ref{fig:transient} validates the finite-noise-level
predictions of Theorem~\ref{thm:risk-general}.  The
theoretical risk curve, computed from the Stieltjes
fixed-point system~\eqref{eq:ridge-risk}, is overlaid on the
empirical risk from 500-trial Monte Carlo simulations across
25 values of the effective noise variance $\sigma_n^2$.  The
median relative error between theory and simulation is $0.03\%$.

\begin{figure}[tp]
\centering
\includegraphics[width=0.75\textwidth]{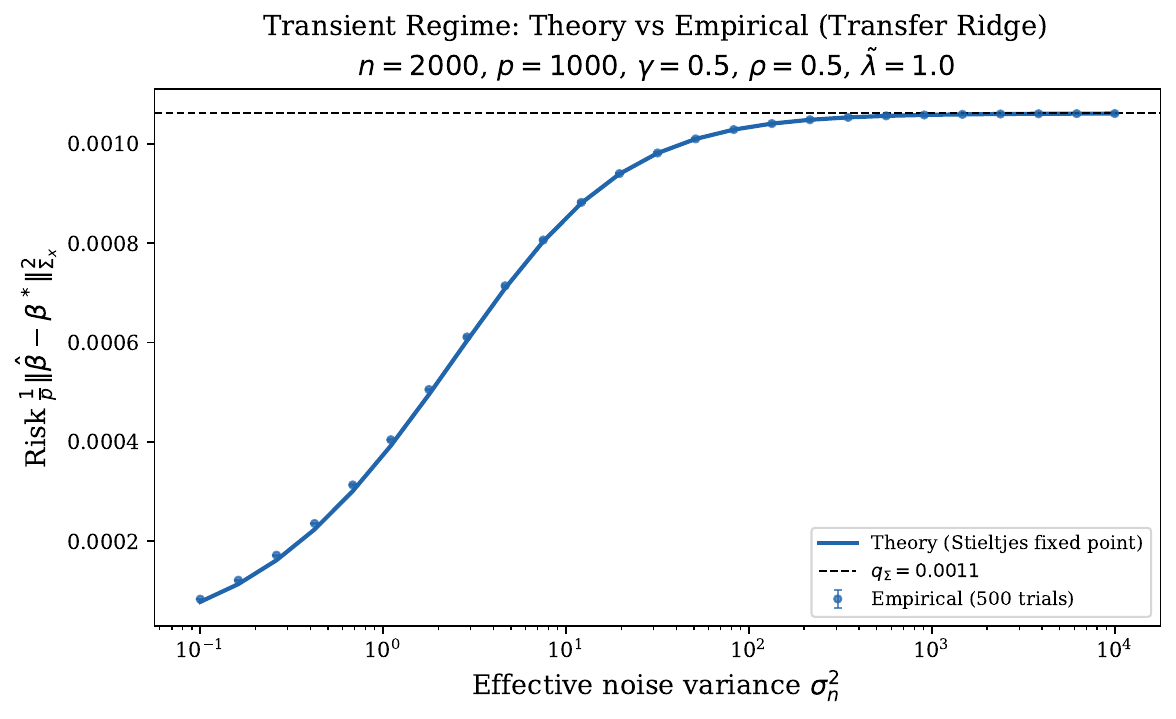}
\caption{Finite-$\sigma_n$ risk of transfer ridge: theoretical
prediction (solid) versus Monte Carlo simulation (dots, 500
trials each).  Median relative error: $0.03\%$.}
\label{fig:transient}
\end{figure}

Figure~\ref{fig:trichotomy} displays the complete trichotomy
of Theorem~\ref{thm:trichotomy}.  The squared-loss estimators
(OLS, fixed ridge) diverge; the Huber estimator
($k = 1.5$, $\lambda = 0.1$, no transfer prior) plateaus at
$\mathcal{R}_H \approx 0.19$; and the transfer ridge estimator
achieves $q_\Sigma \approx 0.001$.  The ratio
$\mathcal{R}_H / q_\Sigma \approx 179$ quantifies the gap
between the two finite-risk regimes.

\begin{figure}[tp]
\centering
\includegraphics[width=0.75\textwidth]{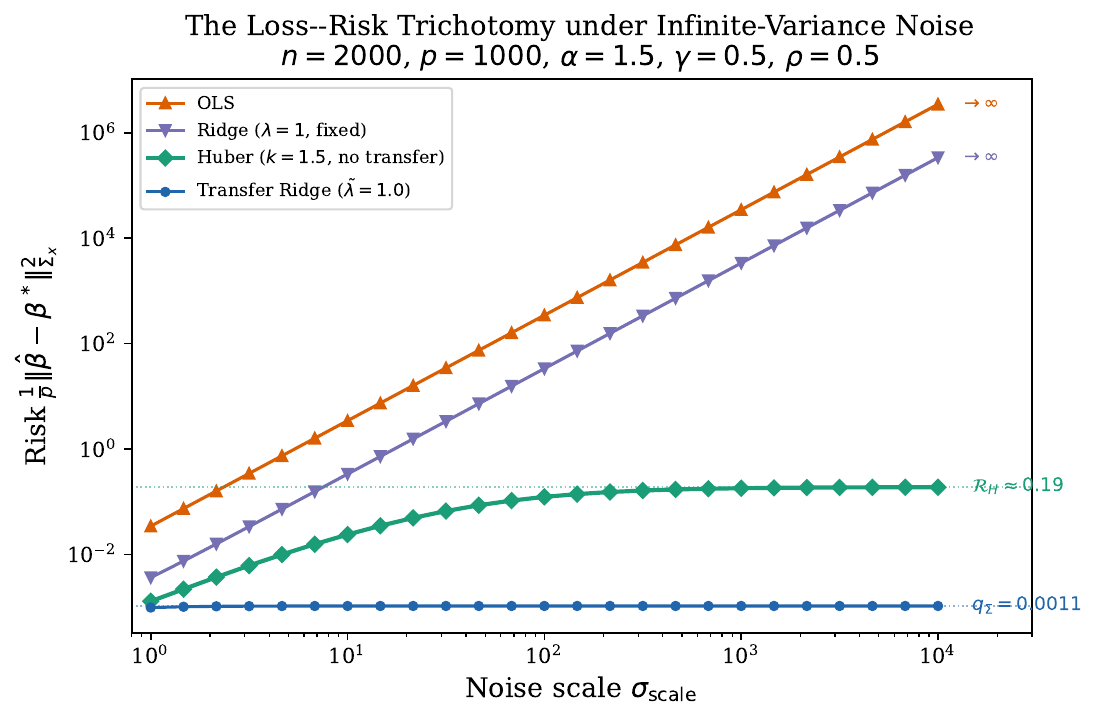}
\caption{The loss--risk trichotomy.  OLS and fixed ridge
diverge; the Huber estimator plateaus at
$\mathcal{R}_H \approx 0.19$; transfer ridge achieves
$q_\Sigma \approx 0.001$.}
\label{fig:trichotomy}
\end{figure}

\section{Discussion}\label{sec:discussion}

\subsection{The two-penalty decomposition}

The trichotomy of Theorem~\ref{thm:trichotomy} reveals that
two structurally distinct penalties impede estimation under
infinite-variance noise.  The \emph{heavy-tail penalty} is a
property of the loss--noise interaction: it arises when the
conjugate domain of the loss is unbounded, exposing the
estimator to a diverging noise moment.  It is eliminated by
adopting any loss function with bounded conjugate domain.
The \emph{high-dimensional penalty} $\mathcal{R}_H > 0$ is
a property of the proportional regime $p/n \to \gamma > 0$:
it persists even when the effective noise is bounded and
reflects the fundamental cost of estimating $p$ parameters
from $n$ observations.  This penalty is bypassed only by
incorporating external structural information through a
transfer prior.

This decomposition provides a reusable diagnostic: when a
practitioner observes poor estimation performance under
heavy-tailed data, the first question is whether the degradation
originates from the loss function (addressable by changing to
a robust loss) or from the high-dimensional geometry
(addressable only through additional information).

\subsection{Open problems}

Several directions for future investigation emerge from this
work.

\emph{Optimal Huber threshold.}
The risk $\mathcal{R}_H(k)$ depends on the threshold
parameter $k$ of the Huber loss.  Characterizing the optimal
threshold $k^*(\alpha, \gamma)$ as a function of the tail
index and the aspect ratio is an important optimization
problem with direct practical implications.

\emph{Non-convex regularizers.}
The CGMT framework requires convexity of both the loss and the
regularizer.  Extending the analysis to non-convex penalties
such as SCAD or MCP, which are widely used in practice, would
require alternative tools.

\emph{Heavy-tailed covariates.}
When the design matrix itself has heavy-tailed entries, the
sub-Gaussian assumption (Theorem~\ref{thm:universality}) is
violated.  Developing comparison inequalities for
heavy-tailed random matrices remains a substantial open
problem.

\emph{Intermediate growth rates.}
The moment hierarchy of Remark~\ref{rmk:moment-hierarchy}
predicts that losses with conjugate growth $|u|^q$ for
$q \in (1,2)$ interpolate between the Huber and squared-loss
behaviors.  A quantitative verification of this prediction for
specific loss functions is a natural next step.

\section*{Acknowledgments}
The author thanks the anonymous referees for their feedback.
Simulation code is available upon request.

\appendix

\section{Truncation universality: full details}
\label{app:truncation}

The complete proof of Theorem~\ref{thm:truncation}(ii),
including the double-limit argument on the triangular array,
is given in the main text (Section~\ref{sec:truncation}).
We provide here additional details on the Karamata bound
used in the tail-part estimate.

For the truncated variables $Y_{n,i} = (w_i^{(\tau_n)})^2
/ \sigma_n^2 - 1$, the tail expectation
$\E[|Y_{n,1}|\ind_{\{|Y_{n,1}| > M\}}]$ is bounded using
the structure of the regularly-varying tail.  Setting
$a_n(M) = ((M-1)\sigma_n^2)^{1/2}$, Karamata's theorem
applied to the ratio of partial integrals gives
$\int_{a_n}^{\tau_n} 2t\bar{F}(t)\,dt
/ \int_0^{\tau_n} 2t\bar{F}(t)\,dt
\to (a_n/\tau_n)^{2-\alpha}$,
which vanishes as $n \to \infty$ for each fixed $M$ since
$a_n/\tau_n \asymp M^{1/2}n^{-1/2} \to 0$.  This establishes
that $\limsup_n \E[|Y_{n,1}|\ind_{\{|Y_{n,1}|>M\}}] \le h(M)$
with $h(M) \to 0$.

\section{CGMT reduction and compactification}
\label{app:cgmt}

The CGMT framework of \citet{Thrampoulidis2015} requires three
conditions: (i)~the feasible set is compact (provided by
Lemma~\ref{lem:compact}); (ii)~the objective is convex--concave
(by construction); (iii)~the random matrix has i.i.d.\ Gaussian
entries.  The noise vector $w^{(\tau_n)}$ enters the Primary
Optimization as a fixed additive vector and is identical in
both the PO and the Auxiliary Optimization; it does not
interact with Gordon's comparison inequality.

The compactification proof (Lemma~\ref{lem:compact}) is given
in full in Section~\ref{sec:proofs-cdd}.  The structural
observation that the design's quadratic confinement vanishes
under the $\sigma_n^2$ normalization is discussed there.

\section{AO convergence and the universal floor}
\label{app:ao}

The AO scalarization and the derivation of the fixed-point
system proceed as follows.  After optimizing over the dual
variable $u$ in closed form, the AO reduces to a minimization
over $v$ involving three empirical quantities:
$n^{-1}\|w^{(\tau_n)}\|^2/\sigma_n^2$,
$n^{-1}\|g\|^2$, and
$g^\top w^{(\tau_n)}/({\sqrt{n}\sigma_n})$.
All three concentrate (the first by
Theorem~\ref{thm:truncation}(ii), the second by the law of
large numbers, the third by conditional Gaussianity), so the
empirical AO converges uniformly on $\mathcal{S}_C$ to a
deterministic limit.  The minimizer of the deterministic limit
satisfies the fixed-point system~\eqref{eq:fp}.

The proof of the universal risk floor
(Theorem~\ref{thm:floor}) is given in full in
Section~\ref{sec:floor}.

The fixed-point uniqueness at leading order follows from the
observation that $\mathcal{R}(\tau) \to q_\Sigma$ uniformly
in $\tau$ on compact sets (since the proximal operator with
diverging step size kills the $\tau$-dependence), so the
equation $\tau^2 = 1 + \gamma q_\Sigma$ has the unique
positive solution $\tau^* = \sqrt{1 + \gamma q_\Sigma}$.

\section{Conjugate Domain Dichotomy: proof details}
\label{app:cdd}

The proof of Theorem~\ref{thm:cdd} is presented in
Section~\ref{sec:proofs-cdd}.  We provide here the
verification that $n^{-1}\sum|w_i^{(\tau_n)}|$ converges.

Since $|w_i^{(\tau_n)}| \le |w_i|$ pointwise and
$\E[|w|] < \infty$ (as $\alpha > 1$), the dominated
convergence theorem gives
$\E[|w^{(\tau_n)}|] \to \E[|w|]$.  By the strong law of large
numbers for the i.i.d.\ bounded random variables
$|w_i^{(\tau_n)}|$ (bounded by $\tau_n$, with convergent
expectations):
$n^{-1}\sum_{i=1}^n |w_i^{(\tau_n)}|
\xrightarrow{\text{a.s.}} \E[|w|] < \infty$.
This ensures that the noise coupling under bounded conjugate
domain is $\mathcal{O}(1)$.

\section{Design universality: condition verification}
\label{app:universality}

We verify the five conditions of
\citet[Corollary~3]{MontanariSaeed2022}:

\begin{enumerate}[label=(\arabic*)]
  \item \textit{Loss function.}
  The squared loss $\ell(u;y) = (y-u)^2/2$ has locally
  Lipschitz gradient $\nabla_u\ell = -(y-u)$, satisfying
  Assumption~1$'$ of \citet{MontanariSaeed2022}.

  \item \textit{Compact constraint set.}
  $\mathcal{S}_C = \{\beta : \|\beta-\beta^*\|_{\Sigma_x}
  \le C\}$ is compact by Lemma~\ref{lem:compact}.

  \item \textit{Regularizer.}
  After $\sigma_n^2$ normalization,
  $\tilde\lambda R(\beta-\beta_0)$ is locally Lipschitz
  on bounded sets with $\mathcal{O}(1)$ constant.

  \item \textit{Sub-Gaussian design.}
  $x_i = \Sigma_x^{1/2}\bar{x}_i$ with i.i.d.\ sub-Gaussian
  $\bar{x}_{ij}$ satisfies pointwise normality
  (Assumption~5 of \citealt{MontanariSaeed2022}).

  \item \textit{Noise independence.}
  $w^{(\tau_n)}$ is independent of $X$ and bounded by
  $\tau_n$ (hence sub-Gaussian).  The universality acts on
  the design, not the noise.
\end{enumerate}

By \citet[Theorem~1]{MontanariSaeed2022}, the training error
is universal.  By \citet[Theorem~3(a)]{MontanariSaeed2022},
the test error is universal for strongly convex regularizers;
the Lasso case follows by $\varepsilon$-perturbation.

\bibliographystyle{plainnat}
\bibliography{references}

@book{BinghamGoldieTeugels1987,
  author    = {Bingham, N. H. and Goldie, C. M. and Teugels, J. L.},
  title     = {Regular Variation},
  publisher = {Cambridge University Press},
  series    = {Encyclopedia of Mathematics and its Applications},
  volume    = {27},
  year      = {1987},
  doi       = {10.1017/CBO9780511721434},
}

@book{Resnick2007,
  author    = {Resnick, Sidney I.},
  title     = {Heavy-Tail Phenomena: Probabilistic and Statistical Modeling},
  publisher = {Springer},
  year      = {2007},
  doi       = {10.1007/978-0-387-45024-7},
}

@book{SamorodnitskyyTaqqu1994,
  author    = {Samorodnitsky, Gennady and Taqqu, Murad S.},
  title     = {Stable Non-{G}aussian Random Processes},
  publisher = {Chapman \& Hall},
  year      = {1994},
}

@article{Huber1964,
  author    = {Huber, Peter J.},
  title     = {Robust estimation of a location parameter},
  journal   = {The Annals of Mathematical Statistics},
  volume    = {35},
  number    = {1},
  pages     = {73--101},
  year      = {1964},
  doi       = {10.1214/aoms/1177703732},
}

@book{HuberRonchetti2009,
  author    = {Huber, Peter J. and Ronchetti, Elvezio M.},
  title     = {Robust Statistics},
  edition   = {2nd},
  publisher = {John Wiley \& Sons},
  year      = {2009},
  doi       = {10.1002/9780470434697},
}

@article{ElKarouiBeanBickelEtAl2013,
  author    = {El~Karoui, Noureddine and Bean, Derek and Bickel, Peter J.
               and Lim, Chinghway and Yu, Bin},
  title     = {On robust regression with high-dimensional predictors},
  journal   = {Proceedings of the National Academy of Sciences},
  volume    = {110},
  number    = {36},
  pages     = {14557--14562},
  year      = {2013},
  doi       = {10.1073/pnas.1307842110},
}

@inproceedings{Thrampoulidis2015,
  author    = {Thrampoulidis, Christos and Oymak, Samet and Hassibi, Babak},
  title     = {Regularized linear regression: A precise analysis of the
               estimation error},
  booktitle = {Proceedings of the 28th Conference on Learning Theory (COLT)},
  pages     = {1683--1709},
  year      = {2015},
}

@article{ThrampoulidisOymakHassibi2018,
  author    = {Thrampoulidis, Christos and Oymak, Samet and Hassibi, Babak},
  title     = {Precise error analysis of regularized {M}-estimators in
               high dimensions},
  journal   = {IEEE Transactions on Information Theory},
  volume    = {64},
  number    = {8},
  pages     = {5592--5628},
  year      = {2018},
  doi       = {10.1109/TIT.2018.2840720},
}

@inproceedings{Gordon1988,
  author    = {Gordon, Yehoram},
  title     = {On {M}ilman's inequality and random subspaces which escape
               through a mesh in $\mathbb{R}^n$},
  booktitle = {Geometric Aspects of Functional Analysis},
  series    = {Lecture Notes in Mathematics},
  volume    = {1317},
  pages     = {84--106},
  publisher = {Springer},
  year      = {1988},
}

@inproceedings{MontanariSaeed2022,
  author    = {Montanari, Andrea and Saeed, Basil},
  title     = {Universality of empirical risk minimization},
  booktitle = {Proceedings of the 35th Conference on Learning Theory (COLT)},
  series    = {PMLR},
  volume    = {178},
  year      = {2022},
}

@article{HuLu2023,
  author    = {Hu, Hong and Lu, Yue M.},
  title     = {Universality laws for high-dimensional learning with
               random features},
  journal   = {IEEE Transactions on Information Theory},
  volume    = {69},
  number    = {3},
  pages     = {1932--1964},
  year      = {2023},
  doi       = {10.1109/TIT.2022.3217698},
}

@article{DobribanWager2018,
  author    = {Dobriban, Edgar and Wager, Stefan},
  title     = {High-dimensional asymptotics of prediction: Ridge regression
               and classification},
  journal   = {The Annals of Statistics},
  volume    = {46},
  number    = {1},
  pages     = {247--279},
  year      = {2018},
  doi       = {10.1214/17-AOS1549},
}

@article{HastieMontanariRossetTibshirani2022,
  author    = {Hastie, Trevor and Montanari, Andrea and Rosset, Saharon
               and Tibshirani, Ryan J.},
  title     = {Surprises in high-dimensional ridgeless least squares
               interpolation},
  journal   = {The Annals of Statistics},
  volume    = {50},
  number    = {2},
  pages     = {949--986},
  year      = {2022},
  doi       = {10.1214/21-AOS2133},
}

@article{BartlettLongLugosiTsigler2020,
  author    = {Bartlett, Peter L. and Long, Philip M. and Lugosi, G\'{a}bor
               and Tsigler, Alexander},
  title     = {Benign overfitting in linear regression},
  journal   = {Proceedings of the National Academy of Sciences},
  volume    = {117},
  number    = {48},
  pages     = {30063--30070},
  year      = {2020},
  doi       = {10.1073/pnas.1907378117},
}

@article{TsiglerBartlett2023,
  author    = {Tsigler, Alexander and Bartlett, Peter L.},
  title     = {Benign overfitting in ridge regression},
  journal   = {Journal of Machine Learning Research},
  volume    = {24},
  number    = {123},
  pages     = {1--76},
  year      = {2023},
}

@article{MeiMontanari2022,
  author    = {Mei, Song and Montanari, Andrea},
  title     = {The generalization error of random features regression},
  journal   = {Communications on Pure and Applied Mathematics},
  volume    = {75},
  number    = {4},
  pages     = {667--766},
  year      = {2022},
  doi       = {10.1002/cpa.22008},
}

@article{BelkinHsuMaMandal2019,
  author    = {Belkin, Mikhail and Hsu, Daniel and Ma, Siyuan and Mandal, Soumik},
  title     = {Reconciling modern machine-learning practice and the classical
               bias--variance trade-off},
  journal   = {Proceedings of the National Academy of Sciences},
  volume    = {116},
  number    = {32},
  pages     = {15849--15854},
  year      = {2019},
  doi       = {10.1073/pnas.1903070116},
}

@article{AdvaniSaxe2020,
  author    = {Advani, Madhu S. and Saxe, Andrew M.},
  title     = {High-dimensional dynamics of generalization error in
               neural networks},
  journal   = {Neural Networks},
  volume    = {132},
  pages     = {428--446},
  year      = {2020},
  doi       = {10.1016/j.neunet.2020.08.022},
}

@article{DarBaraniuk2022,
  author    = {Dar, Yehonathan and Baraniuk, Richard G.},
  title     = {Double double descent: On generalization errors in transfer
               learning between linear regression tasks},
  journal   = {SIAM Journal on Mathematics of Data Science},
  volume    = {4},
  number    = {4},
  pages     = {1447--1472},
  year      = {2022},
  doi       = {10.1137/21M1458788},
}

@article{LiCaiLi2022,
  author    = {Li, Sai and Cai, T. Tony and Li, Hongzhe},
  title     = {Transfer learning for high-dimensional linear regression},
  journal   = {Journal of the Royal Statistical Society: Series B},
  volume    = {84},
  number    = {1},
  pages     = {149--173},
  year      = {2022},
  doi       = {10.1111/rssb.12479},
}

@article{Adomaityte2024,
  author    = {Adomaityt\.{e}, Urte and Defilippis, Emanuele and
               Loureiro, Bruno and Sicuro, Gabriele},
  title     = {High-dimensional robust regression under heavy-tailed data},
  journal   = {arXiv preprint arXiv:2309.16476v2},
  year      = {2024},
  eprint    = {2309.16476},
  archiveprefix = {arXiv},
}

@article{LugosiMendelson2019,
  author    = {Lugosi, G\'{a}bor and Mendelson, Shahar},
  title     = {Mean estimation and regression under heavy-tailed
               distributions: A survey},
  journal   = {Foundations of Computational Mathematics},
  volume    = {19},
  number    = {5},
  pages     = {1145--1190},
  year      = {2019},
  doi       = {10.1007/s10208-019-09427-x},
}

@article{MarchenkoPastur1967,
  author    = {Marchenko, Vladimir A. and Pastur, Leonid A.},
  title     = {Distribution of eigenvalues for some sets of random matrices},
  journal   = {Matematicheskii Sbornik},
  volume    = {1},
  number    = {4},
  pages     = {457--483},
  year      = {1967},
  doi       = {10.1070/SM1967v001n04ABEH001994},
}

@book{BaiSilverstein2010,
  author    = {Bai, Zhidong and Silverstein, Jack W.},
  title     = {Spectral Analysis of Large Dimensional Random Matrices},
  edition   = {2nd},
  publisher = {Springer},
  year      = {2010},
  doi       = {10.1007/978-1-4419-0661-8},
}

@book{BauschkeCombettes2017,
  author    = {Bauschke, Heinz H. and Combettes, Patrick L.},
  title     = {Convex Analysis and Monotone Operator Theory in
               {H}ilbert Spaces},
  edition   = {2nd},
  publisher = {Springer},
  year      = {2017},
  doi       = {10.1007/978-3-319-48311-5},
}

\end{document}